\documentclass{amsart}
\usepackage{graphicx,tikz-cd}
\usepackage[T1]{fontenc}
\usepackage[new]{old-arrows}
\usepackage{amssymb,amsthm}
\usepackage{hyperref}

\usepackage{enumitem}
\usepackage{listlbls}
\newcommand{\R}{\mathbb R}

\newcommand{\xtilde}{\widetilde{X}}

\newtheorem{theorem}{Theorem}[section] 
\newtheorem*{theorem*}{Theorem}
\newtheorem{lemma}[theorem]{Lemma}

\newtheorem*{conjecture*}{Conjecture}

\newtheorem*{question*}{Question}
\newtheorem*{lemma*}{Lemma}
\newtheorem{proposition}[theorem]{Proposition}

\newtheorem*{corollary*}{Corollary}

\theoremstyle{definition}

\newtheorem{remark}[theorem]{Remark}

\newtheorem*{remark*}{Remark}

\newtheorem*{example*}{Example}
\newtheorem*{remarks*}{Remarks}
\newtheorem*{addenda*}{Addenda}
\newtheorem*{construction*}{Construction}

\newtheorem*{ques*}{Question}

\theoremstyle{definition}

\title{Exotic aspherical 4-manifolds}
\author[Davis]{Michael Davis}
\address{Department of Mathematics, Math Tower
\newline\indent Ohio State University
\newline\indent Columbus, OH 43210}
\email{davis.12@osu.edu}
\author[Hayden]{Kyle Hayden}
\address{Department of Mathematics \& Computer Science
\newline \indent Rutgers University-Newark \newline \indent Newark, NJ 07102}
\email{kyle.hayden@rutgers.edu}
\author[Huang]{Jingyin Huang}
\address{Department of Mathematics, Math Tower
\newline\indent Ohio State University
\newline\indent Columbus, OH 43210}
\email{huang.929@osu.edu}
\author[Ruberman]{Daniel Ruberman}
\address{Department of Mathematics, MS 050\newline\indent Brandeis
University \newline\indent Waltham, MA 02454
}
\email{ruberman@brandeis.edu}
\author[Sunukjian]{Nathan Sunukjian}
\address{Department of Mathematics and Statistics\newline\indent Calvin University\newline\indent Grand Rapids, MI 49506}
\email{nss9@calvin.edu}
\thanks{DR was partially supported by NSF grant DMS-1952790.  KH was partially supported by DMS-2243128. JH was partially supported by a Sloan fellowship and NSF grant DMS-2305411.}
\date{\today}

\begin{document}

\begin{abstract}
We prove that there exist closed, aspherical, smooth 4-manifolds that are homeomorphic but not diffeomorphic. These provide counterexamples to a smooth analog of the Borel conjecture in dimension four.
\end{abstract}

\maketitle

\section{Introduction}
The Borel conjecture predicts that closed aspherical manifolds are topologically rigid, with their homeomorphism type  determined by their homotopy type. This conjecture seeks to generalize the rigidity exhibited by hyperbolic manifolds and other (nonpositively curved) locally symmetric spaces, a setting in which every homotopy equivalence is realized by an isometry \cite{mostow,mostow:book}.

This paper concerns  the smooth version of the Borel conjecture, which asks whether closed, aspherical $n$-manifolds that are homotopy equivalent are in fact diffeomorphic.\footnote{See~\cite{weinberger:borel} for some interesting historical comments on the formulation of the Borel conjecture as a statement about homeomorphism, rather than diffeomorphism, of homotopy equivalent aspherical manifolds.} 
This is classical in dimensions $n \leq 2$ and is known to hold for orientable $3$-manifolds, using Perelman's results~\cite{perelman1,perelman2,perelman3,morgan-tian:geometrization,besson-etal:geometrisation}. On the other hand, it is known~\cite[Chapter 15]{wall:surgery} that there exist exotic aspherical manifolds in dimensions at least $5$, including exotic smooth structures on $T^n$ for $n\geq 5$. In dimensions at least $7$, such examples can be obtained by connected sum of certain aspherical manifolds (such as tori or stably parallelizable hyperbolic manifolds) with an exotic sphere~\cite{bustamante-tshishiku:symmetries,farrell-jones:exotic}. We resolve the last remaining case, in dimension 4:

\begin{theorem}\label{thm:exotic-mfld}
There exist pairs of smooth, closed, aspherical 4-manifolds that are homeomorphic but not diffeomorphic.
\end{theorem}
In fact, we will find infinitely many such pairs; see Remark~\ref{rem:infinite}.
The key adjective in the theorem is the word \emph{closed}.
It was previously known that there are exotic smooth structures on $\R^4$ (cf.~\cite{gompfstipsicz}) and on compact aspherical manifolds with boundary, in fact on compact contractible manifolds~\cite{akbulut-ruberman:absolute}. (If one is working relative to a fixed identification of the boundary, as in some formulations of the Borel conjecture, then exotic contractible manifolds go back to~\cite{akbulut:cork}.) 

Among closed manifolds, most previous constructions of exotic 4-manifolds appear ill-suited to the aspherical setting. For example, one obstacle is that closed, aspherical 4-manifolds must have infinite fundamental groups, and these are typically not known to be ``good'' groups that allow~\cite{freedman-quinn} topological surgery theory and the s-cobordism theorem to function in the $4$-dimensional setting. Therefore, one must take a more local approach to establishing homeomorphisms; our exotic pairs $Q,Q'$ are related by explicit cork twisting, which is known to preserve homeomorphism type by work of Freedman \cite{freedman}. Since cork twisting takes place in a contractible region, it has more in common with counterexamples obtained by sums with exotic spheres than with other constructions from surgery theory.

\begin{figure}
\center
 \def\svgwidth{\linewidth}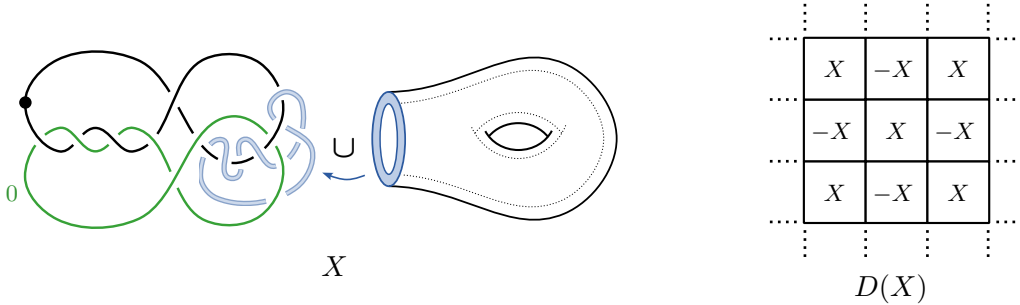 
\caption{\textsc{(Left)} The 4-manifold $X$ is the union of a contractible 4-manifold and a thickened, once-punctured torus. \textsc{(Right)} The 4-manifold $D(X)$ is a union of countably many copies of $X$ and $-X$, glued along 3-cells in their boundaries.}\label{fig:schematic}
\end{figure}

The construction underlying Theorem~\ref{thm:exotic-mfld} begins with a pair of exotic aspherical 4-manifolds $X,X'$ with boundary (based on \cite{hp:embedding}) and then produces closed 4-manifolds by applying the reflection group trick developed by the first author \cite{davis:aspherical}.  
As depicted schematically in Figure~\ref{fig:schematic},  each of $X$ and $X'$ is obtained from a contractible 4-manifold $C$ (namely the Akbulut cork \cite{akbulut:cork}) by attaching a thickened punctured torus, hence has the homotopy type of $T^2$, which is aspherical. 
Necessary data for the reflection group trick includes a triangulation $\mathcal{T}$ of $\partial X$ as a flag complex.  The $1$-skeleton of $\mathcal{T}$ determines a right-angled Coxeter group $W(\mathcal{T})$.
The reflection group trick proceeds by constructing an associated  noncompact space $D(X)$ (resp.,~$D(X')$) built from infinitely many copies of $X$ and $-X$ (resp.,~$X'$ and $-X'$), as depicted schematically on the right-hand side of Figure~\ref{fig:schematic}. The closed aspherical 4-manifolds $Q(X)$ and $Q(X')$ claimed in Theorem~\ref{thm:exotic-mfld} are then obtained as certain quotients of $D(X)$ and $D(X')$. 

The main claim in the theorem is then proved in two steps. In  \S\ref{sec:exotic}, we construct a homeomorphism between $Q(X)$ and $Q(X')$, and argue that it is not homotopic to a diffeomorphism.\footnote{In many treatments, e.g.,~\cite{farrell:borel,weinberger:borel}, the Borel conjecture is stated as saying that any homotopy equivalence is homotopic to a homeomorphism, so that this step would already give a counterexample to that version of the smooth Borel conjecture.} In fact, we show that there is no diffeomorphism between $Q(X)$ and $Q(X')$ that lifts to the covering spaces $D(X)$ and $D(X')$. (These covers are distinguished by comparing the genera of smoothly embedded surfaces; here some care with obstructions is required because these covers are built from copies of $X$ with both orientations.) This reduces Theorem~\ref{thm:exotic-mfld} to a lifting problem that we solve in \S\ref{sec:characteristic}; the key is an algebraic argument showing that the fundamental group $\pi_1(D(X))$ (resp.,~$\pi_1(D(X'))$) is characteristic in $\pi_1(Q(X))$ (resp.~$\pi_1(Q(X'))$). In order to make this algebraic argument work we need to make a special choice of $\mathcal{T}$: it must satisfy the ``flag-no-square condition.'' This implies that the resulting Coxeter group is word hyperbolic, a fact that we need in our algebraic argument.
The fact that $\pi_1(D(X))$ and $\pi_1(D(X'))$ are characteristic subgroups implies that $Q(X)$ and $Q(X')$ are not diffeomorphic.

We close this discussion by noting that the reflection group trick has had many applications to the construction of ``exotic'' spaces and groups, including closed aspherical manifolds that are not covered by Euclidean space \cite{davis:aspherical},  Poincar\'e duality groups that are not finitely presented \cite{davis:cohomology} (cf.~\cite{bestvina-brady}), 4-dimensional locally CAT(0)-manifolds that do not admit a Riemannian metric of nonpositive curvature \cite{djl:cat(0)}, and --- closer to our purposes here --- aspherical topological manifolds that admit no smooth structure \cite{davis-hausmann}. However, to our knowledge, Theorem~\ref{thm:exotic-mfld} is the first application of the reflection group trick to the study of exotic smooth structures on manifolds. Given the trick's capacity for promoting exotic phenomena from the compact to the closed setting, we expect further applications in this direction. 
     
\noindent 
\subsection*{Acknowledgments} The authors thank Lisa Piccirillo for valuable discussions at the beginning of this project. We also thank Jim Davis and Bena Tshishiku for corrections to some of the introductory historical remarks. Our further thanks go to the referee for helpful comments. 

\section{The input manifolds}\label{sec:input}

Let  $X$ be the 4-manifold shown in Figure~\ref{fig:input-1}, which arises from a very slight modification to the examples underlying \cite[Theorem~4.1]{hp:embedding}. The manifold $X$ is obtained from the contractible Akbulut cork $C$ \cite{akbulut:cork} by attaching a ``genus-1 handle'' (i.e., a copy of $F \times D^2$ where $F$ is a genus-1 surface with one boundary component) along a knot in $\partial C$. The embedded copy $C \subset X$ can be seen as the union of the $0$-handle, the 1-handle represented by the topmost dotted curve, and the inner 2-handle. 

Our proof of Theorem~\ref{thm:exotic-mfld} will leverage the following properties of $X$. 

\begin{figure}
\center
 \def\svgwidth{\linewidth}
\begingroup%
  \makeatletter%
  \providecommand\color[2][]{%
    \errmessage{(Inkscape) Color is used for the text in Inkscape, but the package 'color.sty' is not loaded}%
    \renewcommand\color[2][]{}%
  }%
  \providecommand\transparent[1]{%
    \errmessage{(Inkscape) Transparency is used (non-zero) for the text in Inkscape, but the package 'transparent.sty' is not loaded}%
    \renewcommand\transparent[1]{}%
  }%
  \providecommand\rotatebox[2]{#2}%
  \newcommand*\fsize{\dimexpr\f@size pt\relax}%
  \newcommand*\lineheight[1]{\fontsize{\fsize}{#1\fsize}\selectfont}%
  \ifx\svgwidth\undefined%
    \setlength{\unitlength}{872.81861313bp}%
    \ifx\svgscale\undefined%
      \relax%
    \else%
      \setlength{\unitlength}{\unitlength * \real{\svgscale}}%
    \fi%
  \else%
    \setlength{\unitlength}{\svgwidth}%
  \fi%
  \global\let\svgwidth\undefined%
  \global\let\svgscale\undefined%
  \makeatother%
  \begin{picture}(1,0.45102966)%
    \lineheight{1}%
    \setlength\tabcolsep{0pt}%
    \put(0,0){\includegraphics[width=\unitlength,page=1]{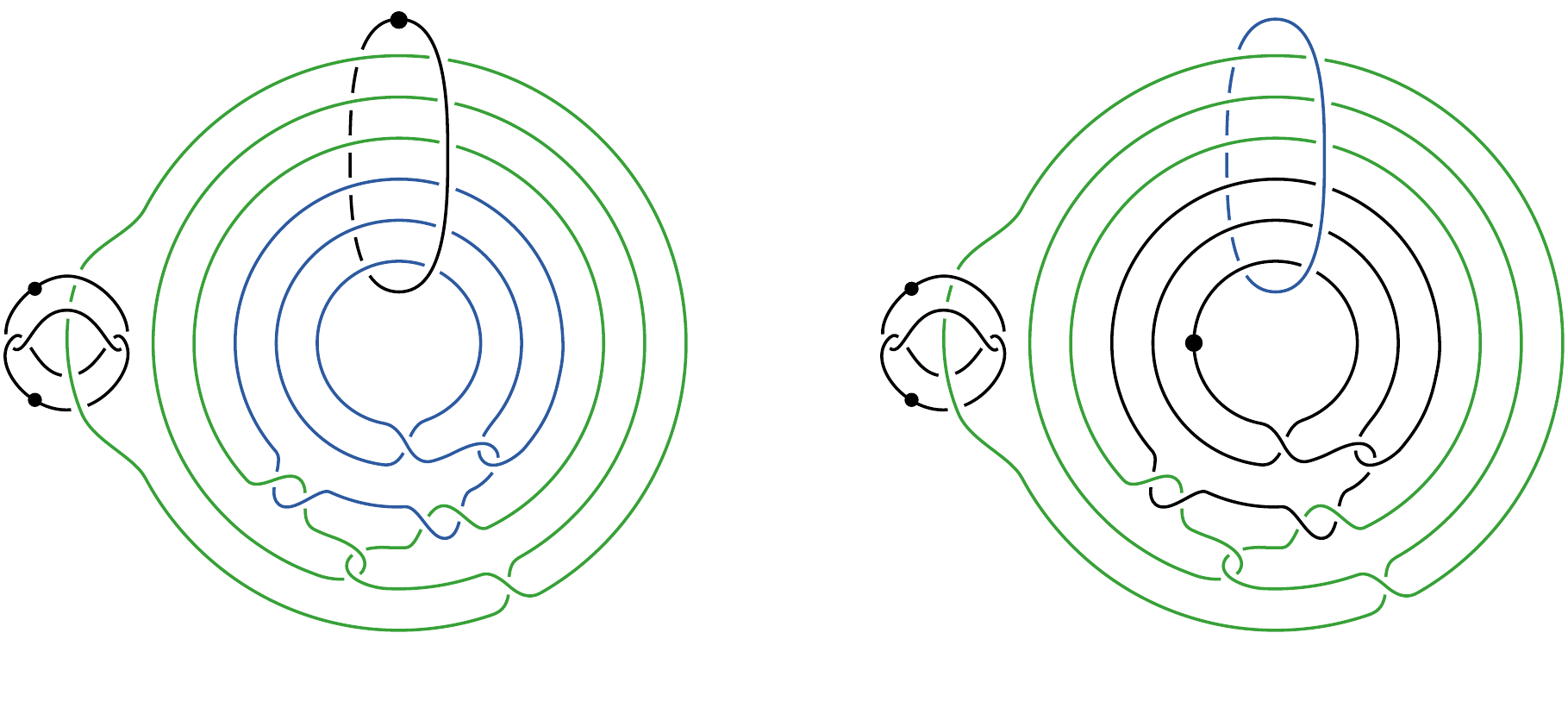}}%
    \put(0.28523325,0.20764837){\color[rgb]{0.17254902,0.35294118,0.62745098}\makebox(0,0)[t]{\smash{\begin{tabular}[t]{c}$0$\end{tabular}}}}%
    \put(0.83966598,0.43551445){\color[rgb]{0.17254902,0.35294118,0.62745098}\makebox(0,0)[t]{\smash{\begin{tabular}[t]{c}$0$\end{tabular}}}}%
    \put(0.42481263,0.1191904){\color[rgb]{0.21568627,0.63137255,0.21568627}\makebox(0,0)[t]{\smash{\begin{tabular}[t]{c}$0$\end{tabular}}}}%
    \put(0.98410461,0.1191904){\color[rgb]{0.21568627,0.63137255,0.21568627}\makebox(0,0)[t]{\smash{\begin{tabular}[t]{c}$0$\end{tabular}}}}%
    \put(0.25449991,0.00553111){\color[rgb]{0,0,0}\makebox(0,0)[t]{\smash{\begin{tabular}[t]{c}$X$\end{tabular}}}}%
    \put(0.81379982,0.00553111){\color[rgb]{0,0,0}\makebox(0,0)[t]{\smash{\begin{tabular}[t]{c}$X'$\end{tabular}}}}%
    \put(0,0){\includegraphics[width=\unitlength,page=2]{aspherical-base-color.pdf}}%
  \end{picture}%
\endgroup%
\caption{Kirby diagrams for the 4-manifolds $X$ and $X'$.}\label{fig:input-1}
\end{figure}

\begin{proposition}\label{prop:properties}
The 4-manifold $X$ satisfies the following:
\begin{enumerate}[label={\normalfont(\alph*)}]
\item $X$ is homotopy equivalent to the torus,
\item $X$ embeds smoothly in $B^4$, 
\item every homologically essential, smoothly embedded surface in $X$ has genus $\geq 2$ (hence the same is true of $-X$), and
\item $X$ is homeomorphic to a smooth 4-manifold $X'$ such that $H_2(X')$ is generated by a smoothly embedded torus.
\end{enumerate}
\end{proposition}

Most of these properties follow verbatim from the proof of \cite[Theorem~4.1]{hp:embedding}; for the reader's convenience, we sketch the arguments and add detail only where necessary. 
 
\begin{proof}
For (a), recall from above that $X$ is obtained from the contractible Akbulut cork $C$ by attaching a genus-1 handle $F \times D^2$. Collapsing $F \times D^2$ to $F \times 0$ and $C$ to a point (hence $\partial F$ to a point) yields a map to $T^2$ that is easily seen to be a weak homotopy equivalence, hence a homotopy equivalence. 

For (b), first attach 0-framed 2-handles along meridians to all of the 1-handle curves of $X$ in Figure~\ref{fig:input-1}. This has the diagrammatic effect of erasing the 1-handles of $X$, leaving only the 0-framed 2-handle curves. These can be seen to form a 2-component unlink, so attaching two 3-handles yields $B^4$.

For (c), consider the handle diagram for $X$ shown in Figure~\ref{fig:stein}, which is in Gompf's standard form \cite{gompf:stein}. It can be checked that the Thurston-Bennequin numbers $tb$ and rotation numbers $r$ of the (oriented) 2-handle curves $B$ and $G$ satisfy
$$tb(B)=1, \ \ tb(G)=1 \quad \qquad \qquad r(B)=1, \ \  r(G)=3.$$

The 2-handle framings both correspond to the integer $0=tb-1$, hence, $X$ admits a Stein structure \cite{gompf:stein}. Observe that $H_2(X)$ is generated by a class $\alpha$ corresponding to the difference of the 2-handles attached along the oriented curves $B$ and $G$.  Gompf's formula \cite[Proposition~2.1]{gompf:stein} for the Chern class $c_1(X)$ of the Stein structure on $X$ yields
$$|\langle c_1(X) , \alpha \rangle| = |r(B)-r(G)| =| 1-3|=2.$$ 

Now suppose that $S$ is a smoothly embedded surface in $X$ satisfying $[S]=k\alpha$ for $k \neq 0$. The adjunction inequality for homologically essential, smoothly embedded surfaces of nonnegative self-intersection in $X$ \cite{lisca-matic} gives us
$$2g(S)-2 \geq \left| \langle c_1(X), [S] \rangle \right| + [S] \cdot [S] = |\langle c_1(X),k\alpha \rangle | + k^2 (\alpha \cdot \alpha)= 2|k|+k^2 \cdot 0,$$
hence $2g(S)\geq 2|k|+2$ and thus $g(S) \geq |k|+1 \geq 2$.

\begin{figure}
\center
 \def\svgwidth{.575\linewidth}
\begingroup%
  \makeatletter%
  \providecommand\color[2][]{%
    \errmessage{(Inkscape) Color is used for the text in Inkscape, but the package 'color.sty' is not loaded}%
    \renewcommand\color[2][]{}%
  }%
  \providecommand\transparent[1]{%
    \errmessage{(Inkscape) Transparency is used (non-zero) for the text in Inkscape, but the package 'transparent.sty' is not loaded}%
    \renewcommand\transparent[1]{}%
  }%
  \providecommand\rotatebox[2]{#2}%
  \newcommand*\fsize{\dimexpr\f@size pt\relax}%
  \newcommand*\lineheight[1]{\fontsize{\fsize}{#1\fsize}\selectfont}%
  \ifx\svgwidth\undefined%
    \setlength{\unitlength}{410.20920641bp}%
    \ifx\svgscale\undefined%
      \relax%
    \else%
      \setlength{\unitlength}{\unitlength * \real{\svgscale}}%
    \fi%
  \else%
    \setlength{\unitlength}{\svgwidth}%
  \fi%
  \global\let\svgwidth\undefined%
  \global\let\svgscale\undefined%
  \makeatother%
  \begin{picture}(1,0.68157579)%
    \lineheight{1}%
    \setlength\tabcolsep{0pt}%
    \put(0,0){\includegraphics[width=\unitlength,page=1]{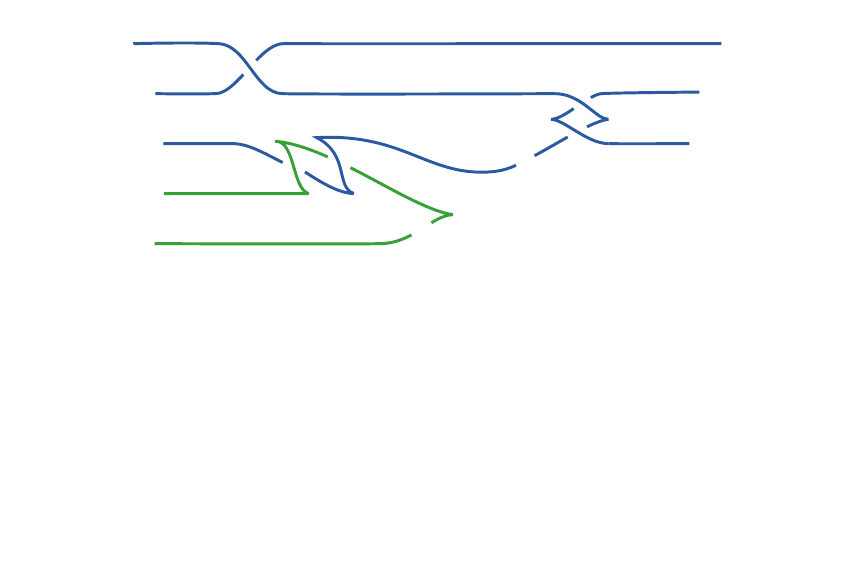}}%
    \put(0.20995687,0.65315811){\color[rgb]{0.17254902,0.35294118,0.62745098}\makebox(0,0)[lt]{\lineheight{1.25}\smash{\begin{tabular}[t]{l}$B$\end{tabular}}}}%
    \put(0.75247974,0.28381561){\color[rgb]{0.21568627,0.63137255,0.21568627}\makebox(0,0)[lt]{\lineheight{1.25}\smash{\begin{tabular}[t]{l}$0$\end{tabular}}}}%
    \put(0.75247974,0.64950148){\color[rgb]{0.17254902,0.35294118,0.62745098}\makebox(0,0)[lt]{\lineheight{1.25}\smash{\begin{tabular}[t]{l}$0$\end{tabular}}}}%
    \put(0.20872489,0.28015899){\color[rgb]{0.21568627,0.63137255,0.21568627}\makebox(0,0)[lt]{\lineheight{1.25}\smash{\begin{tabular}[t]{l}$G$\end{tabular}}}}%
    \put(0,0){\includegraphics[width=\unitlength,page=2]{stein-color.pdf}}%
  \end{picture}%
\endgroup%
 \caption{A Stein handle diagram for $X$.}\label{fig:stein}
\end{figure}

For (d), let $X'$ be the 4-manifold shown on the right-hand side of Figure~\ref{fig:input-1}. It is obtained from $X$ by twisting along the Akbulut cork $C \subset X$, which can be described diagrammatically by exchanging the roles of the inner 0-framed 2-handle and the topmost dotted 1-handle curve. This corresponds to removing $\mathring{C} \subset X$ and regluing  by an involution of $\partial C$ that extends to a homeomorphism of $C$ (but not a diffeomorphism); see \cite{akbulut:cork,freedman}. 

To identify a torus that generates $H_2(X')$, observe that $X'$ is also obtained by adding a genus-1 handle to the Akbulut cork $C$, attached along the curve underlying the outer 2-handle on the right-hand side of Figure~\ref{fig:input-1}. At the cost of dragging along the other 2-handle curve, one can check that this outer curve is an unknot that can be isotoped away from the inner dotted 1-handle curve, hence, it bounds a smooth disk in $C \subset X'$. Capping off this disk with the core genus-1 surface in the genus-1 handle yields the desired torus.
\end{proof}
\begin{remark}\label{rem:extend}
    We can extend this to an infinite family of such exotic pairs $X_m,X_m'$ for $m \geq 0$ (all homotopy equivalent to $T^2$) such that $H_2(X_m')$ is generated by a smoothly embedded torus of self-intersection $0$, but a nonzero element of $H_2(X_m)$ cannot be represented by a smoothly embedded surface $S$ of genus less than $2+m$. This can be achieved by modifying the attaching curve $G$ by adding $m$ positive clasps across the 1-handle and $2m$ positive stabilizations as in \cite[Figure 8]{hp:embedding}. (The framing of $G$ is unchanged.) This has the effect of preserving $tb(G)$ while increasing $r(G)$ by $2m$, which in turn increases the absolute value of the evaluation of the Chern class $c_1(X_m)$ on the generator of $H_2(X_m)$ by $2m$. 
\end{remark}

\section{The reflection group trick}\label{sec:reflection}

We now recall some components of Davis' construction \cite{davis:aspherical,davis:book2}. In this section, $X$ will denote a smooth manifold (of any dimension) with a flag triangulation of its boundary. (Recall that a simplicial complex  is \emph{flag} if any finite subset of its vertices that are pairwise connected by edges spans a simplex. For example, the barycentric subdivision of any simplicial complex is flag.)

\subsection{The Coxeter group}
\label{subsec:coxeter}
Fix a flag triangulation $\mathcal{T}$ of $\partial X$.   Letting $V$ denote the set of vertices in $\mathcal{T}$, there is an associated  Coxeter system $(W,V)$ where $W$ is the right-angled Coxeter group with a generator $v$ of order two for each vertex $v \in V$ and a relation of the form $(vw)^2=1$ for each pair of vertices $v,w$ joined by an edge in $\mathcal{T}$. 

More generally, for any flag simplicial complex $\mathcal{T}$ (not necessarily a manifold), one can define an associated right-angled Coxeter group $W(\mathcal{T})$ in a similar way.

\subsection{The Davis complex} 
\label{subsec:davis space}
If $\mathcal{T}$ is a PL triangulation of $\partial X$, then the dual decomposition $\mathcal{T}'$ to $\mathcal{T}$ is a cell structure on $\partial X$. If $\mathcal{T}$ is a smooth triangulation of $\partial X$, then $X$ is naturally a smooth manifold with corners so that  each stratum in $\partial X$ is a dual cell to a simplex of $\mathcal{T}$. (The top-dimensional cells of $\mathcal{T}'$ are called the \emph{mirrors} of $X$. These are called \emph{panels} in \cite{davis:aspherical}.) In fact, $X$ is a \emph{manifold with faces} as defined in \cite[p.~304]{davis:aspherical} (meaning that each stratum of codimension $k$ is the intersection of precisely $k$ mirrors). Davis' construction yields a noncompact manifold $D(X)$ formed by gluing together a countable collection of copies of $X$ via reflections across the mirrors in the following  prescribed way.  
For each vertex $v$ of $\mathcal T$, let $X_v$ denote the closed cell in $\mathcal {T'}$ dual to $v$. For each point $x\in \partial X$, let $W_x$ denote the subgroup of $W$ generated by all $v\in V$ such that $x\in X_v$. For $x$ in the interior of $X$, the subgroup $W_x$ ($= W_\emptyset$) is defined to be $\{1\}$. Note that $W_x$ is isomorphic to a subgroup of $(\mathbb Z_2)^V$, in particular, each $W_x$ is a finite reflection group isomorphic to $(\mathbb Z_2)^k$ for some $k$. So, $X$ has a \emph{$W$-finite panel structure} (or mirror structure) as defined in  \cite[p.~304]{davis:aspherical}. This gives $X$ the structure of a smooth orbifold where the local group at $x$ is $W_x$. Define an equivalence relation $\sim$ on $W\times X$ by $(h,x)\sim(g,y)$ if and only if $x=y$ and $h^{-1}g\in W_x$. Let $D(X)$ denote the quotient space
$(W \times X)/\!\sim$.  (Note that $D(X)$ was referred to as $\mathfrak{U}(W,X)$ in~\cite{davis:aspherical, davis:book2}.)

Since $D(X)/W$ is the smooth orbifold $X$, the space $D(X)$ is a smooth manifold and the $W$-action is smooth; see \cite[pp.~321-322]{davis:aspherical} or \cite[Remark~10.1.11]{davis:book2}. The Coxeter group $W$ acts smoothly and properly on $D(X)$ with strict fundamental domain $X$, so that we can write $D(X)$ as a union
$$D(X) = \bigcup_{g \in W} g X.$$

Note that the set of generating reflections $V$ gives us a length function $\ell$ on $W$, the word length with respect to $V$. An orientation for $X$ induces an orientation for the manifold $D(X)$ by setting the orientation on each chamber $gX$ to be $(-1)^{\ell(g)}$ times the original orientation of $X$.

We next order the copies of $X$ in a convenient way. Choose any compatible ordering on $W$, $g_1, g_2\dots$ so that $i < j$ implies $\ell(g_i) \leq \ell(g_j)$. This defines an ordering of the chambers: $X_i := g_i X$. Using this ordering, one can write $D(X)$ as an increasing union of subspaces $P_n = \bigcup_{i=1}^n X_i$, where each $P_n$ is  codimension-zero in $D(X)$. (While we can view the subspaces $P_n$ as smooth manifolds with corners, it will suffice to consider them as PL manifolds with boundary.) It turns out that each $P_n$ is a boundary sum of the chambers $X_i$ with $i \leq n$, i.e., $P_n \cong X_1 \natural \cdots \natural \, X_n$. This decomposition is key to our arguments. We sketch its proof in Lemma~\ref{lem:boundary-sum} below. The main input for the proof is the following lemma. 

\begin{lemma}[{\cite[Remark 10.6]{davis:aspherical}}] 
\label{lem:disk}
The intersection of $P_n$ and $X_{n+1}$ is a PL codimension-zero disk $\Delta_n \subset \partial P_n$.
\end{lemma}

\begin{proof}[Sketch of Proof]
The proof of Lemma~\ref{lem:disk} relies on the choice of ordering on $W$. The intersection $P_n\cap X_{n+1}$ is the union of  all mirrors of $X_{n+1}$ such that the reflected image of $X_{n+1}$ across  the mirror  lies in $P_n$. Recall that the mirrors of $X$  correspond to the vertices of the simplicial complex $\mathcal T$.  If $X_{n+1}=gX$, where $g=g_{n+1}$, then the set of such mirrors is indexed by $\mathrm{In}(g):=\{v\in V\mid \ell(gv)<\ell(g)\}$ (cf.~\cite[Lemma 8.1.1]{davis:book2}) -- this is a consequence of the choice of ordering on $W$. The key fact is that for any $g\in W$ the subgroup of $W$ generated by $\mathrm{In}(g)$ is finite (cf.~\cite[Lemma 4.7.2]{davis:book2}). Since the mirror structure is $W$-finite, $\mathrm{In}(g)$ is the vertex set  of a subcomplex of $\mathcal T$ that is isomorphic to a complete graph. Since $\mathcal T$ is a flag complex, $\mathrm{In}(g)$ actually spans a simplex of $\mathcal T$. (This is the only place where the hypothesis that $\mathcal T$ is a flag complex is used.) There is a dual cell to this simplex in $\mathcal T'$. Since $P_n\cap X_{n+1}$ is a regular neighborhood in $\partial X_{n+1}$ of this dual cell, it is a disk $\Delta_n$.
\end{proof}

\begin{lemma}\label{lem:boundary-sum}
For each $n$, there exists a PL homeomorphism $P_n \to X_1 \natural \cdots \natural X_n$.
\end{lemma}

\begin{proof}
We argue inductively, with the lemma being trivially true for $n=1$. Next, suppose that for a given $n$ there exists a PL homeomorphism $$f: P_{n} \to X_1 \natural \cdots \natural X_{n}.$$
By construction, $P_{n+1}= P_n \cup_{\Delta_n} X_{n+1}$, where $\Delta_n \subset \partial P_n$ is the disk $P_n \cap X_{n+1}$ of Lemma~\ref{lem:disk}. It follows that $f$ induces a PL homeomorphism
$$g:  P_{n+1} \to \left(X_1 \natural \cdots \natural X_n\right) \cup_{f(\Delta_n)} X_{n+1},$$
where $f(\Delta_n) \subset \partial (X_1 \natural \cdots \natural \, X_n)$ is identified with the disk in $X_{n+1}$ that was previously identified with $\Delta_n \subset \partial P_n$. 

Now choose a small codimension-zero disk $\Delta'_n \subset \partial X_n$ that lies away from the boundary-summing region in $X_1 \natural \cdots \natural X_{n}$. By the PL version of Palais' disk theorem \cite[Theorem 3.34]{rourke-sanderson}, there is an isotopy of $\partial (X_1 \natural \cdots \natural X_n)$  carrying $f(\Delta_n)$ to $\Delta_n'$, and such an isotopy extends to an isotopy of $X_1 \natural \cdots \natural X_n$ supported near a collar neighborhood of its boundary  (cf.~\cite[Theorem 3.22]{rourke-sanderson}). 
This defines a PL homeomorphism
$$h:  \left(X_1 \natural \cdots \natural X_n\right) \cup_{f(\Delta_n)} X_{n+1} \longrightarrow  \left(X_1 \natural \cdots \natural X_n\right) \cup_{\Delta'_n} X_{n+1},$$
and the latter space is naturally identified with the boundary sum $X_1 \natural \cdots \natural X_{n+1}$.  Composing $g$ and $h$ yields the desired PL homeomorphism from $P_{n+1}$ to $X_1 \natural \cdots \natural X_{n+1}$, completing the inductive argument.
\end{proof}

The description in Lemma~\ref{lem:boundary-sum} introduces two important subtleties. First, when the union $P_n = \bigcup_{i=1}^n X_i$ is expressed as a boundary sum,  we must allow each summand $X_i$ to be diffeomorphic to either $X$ or $-X$. 

The second subtlety is that the  summing region $\Delta_n$ in $P_{n} = P_{n-1} \natural \, X_{n}$ is not confined to the boundary of $X_{n-1} \subset P_{n-1}$, as is the case in the usual construction of $X_1 \natural \cdots \natural X_n$. Therefore, the inductive identification of $P_n$ with $X_1 \natural \cdots \natural \, X_n$ is slightly unnatural.  (This is evidenced by the fact that any given chamber $X_i $ eventually lies strictly in the interior of each $P_n$ for $n \gg i$, yet the $i^{\text{th}}$ summand  of $X_1 \natural \cdots \natural \, X_n$ does not lie in the interior of $X_1 \natural \cdots \natural \, X_n$.)


We record one more simple consequence of Lemma~\ref{lem:disk}.

\begin{lemma}\label{lem:inject}
Each inclusion $X_n \hookrightarrow D(X)$ induces an injection on homology.
\end{lemma}

\begin{proof}
First consider the inclusions $X_n \hookrightarrow P_n$ and $P_{n-1} \hookrightarrow P_n$. For each $n$, observe that these inclusions induce  injections on homology by applying  a Mayer-Vietoris argument to the decomposition $P_n = P_{n-1} \cup X_n$ and by using the fact that the intersection $P_{n-1} \cap X_n$ is a disk. It follows that, for all $m>n$, the composition of inclusions $$X_n \hookrightarrow P_n \hookrightarrow P_{m}$$
induces an injection on homology. Hence, the inclusion $X_n \hookrightarrow D(X)$ induces an injection on homology, since a class $\alpha \in H_*(X_n)$ that becomes null-homologous in $H_*(D(X))$ must become null-homologous in $H_*(P_m)$ for some finite $m>n$.
\end{proof}
 
\subsection{A compact quotient} 
\label{subsec:quotient}
Finally, we produce the closed manifold $Q(X)$. Any finitely generated  Coxeter group $W$  contains a finite index, torsion-free subgroup $W_0$ (see \cite[Corollary 6.12.12]{davis:book2}). Set $Q(X)= D(X)/W_0$. As the action of $W_0$ on $D(X)$ is smooth (as discussed earlier) and free, the projection $D(X)\to Q(X)$ is a covering map and hence $Q(X)$ is a smooth manifold. For example, given that the Coxeter group $W$ is right-angled, one can take $W_0$ to be its commutator subgroup. (The commutator subgroup is the kernel of the natural projection to the abelianization, $W\to (\mathbb Z_2)^V$. This map is injective on each finite subgroup of $W$, so its kernel is torsion-free.)

When $X$ is aspherical (as it is in our case), it is easy to see that $D(X)$ and $Q(X)$ are also aspherical by using Lemma~\ref{lem:boundary-sum}. For $D(X)$, consider a map  $f: S^k \to D(X)$ with $k \geq 2$. Its image is compact, hence lies in some $P_n$. Applying Lemma~\ref{lem:boundary-sum}, we see that $P_n$ is homotopy equivalent to a wedge sum of $n$ copies of $X$. A wedge sum of aspherical spaces is aspherical; hence, $f(S^k)$ is nullhomotopic in $P_n$ and therefore, in $D(X)$. We conclude that $D(X)$ is aspherical, hence, so is $Q(X)$, because the homotopy groups $\pi_k$ of $Q(X)$ and its cover $D(X)$ are isomorphic for $k \geq 2$.

\section{An exotic homeomorphism}\label{sec:exotic}

Let $X$ and $X'$ denote the exotic 4-manifolds from \S\ref{sec:input}. Note that they differ by a cork twist in the interior, so their boundaries are identified, and so we choose the same triangulation for both $\partial X$ and $\partial X'$. Applying the construction from \S\ref{sec:reflection}, we obtain an associated pair of closed, aspherical 4-manifolds $Q(X)$ and $Q(X')$ with covers $D(X)$ and $D(X')$. As a step towards proving Theorem~\ref{thm:exotic-mfld}, we first prove the following:

\begin{theorem}\label{thm:relative}
There is a homeomorphism $Q(X)\to Q(X')$ that is not homotopic to any diffeomorphism.
\end{theorem}
 
 \begin{proof} 
Recall that $X'$ is obtained from $X$ by removing the interior of the Akbulut cork $C \subset \mathring{X}$ and regluing $C$ with a twist. The homeomorphism $X \to X'$ constructed in the proof of Proposition~\ref{prop:properties}(d) can be viewed as being the identity away from $\mathring{C}$, where the definitions of $X$ and $X'$ agree.
This induces a $W$-equivariant homeomorphism $\tilde f: D(X) \to D(X')$, which descends to a homeomorphism $f: Q(X) \to Q(X')$.
 
We claim that there is no diffeomorphism $D(X)\to D(X')$, which will imply that $f: Q(X) \to Q(X')$ is not homotopic to any  diffeomorphism. (By the homotopy lifting property, such a homotopy would lift to a homotopy from $\tilde f: D(X) \to D(X')$ to a diffeomorphism $D(X)\to D(X')$.) To prove this, recall from Proposition~\ref{prop:properties} that $X'$ contains a smoothly embedded torus generating $H_2(X')$. In particular, since each inclusion-induced map $H_2(X_i')\to H_2(D(X'))$ is injective by Lemma~\ref{lem:inject}, we see that $D(X')$ contains smoothly embedded, homologically essential tori.

In contrast, we claim that any smoothly embedded, homologically essential surface $S$ in $D(X)$ has genus at least two. Since $S$ is compact, it must be contained in one of the compact subspaces $P_n$ in the exhaustion of $D(X)$. By Lemma~\ref{lem:boundary-sum}, there is a PL homeomorphism $\varphi: P_n \to X_1 \natural \cdots \natural X_n$. The surface $\varphi(S)$ is a locally flat, PL codimension-2 submanifold of the smooth manifold $X_1 \natural \cdots \natural X_n$, hence, it is isotopic to a smoothly embedded surface $S'$ by Wall \cite{wall:codim2} (also see the proof of \cite[Lemma~A.3]{hom-levine-lidman}).

Since $[S]$ is nonzero in $H_2(D(X))$, it is nonzero in $H_2(P_n)$, hence its image $\varphi_*[S]=[S']$ is nonzero in $H_2(X_1 \natural \cdots \natural X_n)$. Note that $H_2(X_1 \natural \cdots \natural X_n)$ splits as a direct sum of $H_2(X_i)$, giving a projection $H_2(X_1 \natural \cdots \natural X_n)\to H_2(X_i)$ for each $i$. It follows that $[S']$ must project to a nonzero element in the homology $H_2(X_k)$ of at least one summand $X_k$ in  $X_1 \natural \cdots \natural X_n$.  By Proposition~\ref{prop:properties}, we may attach 2- and 3-handles to all the other summands $X_i$ for $i \neq k$ (attached away from the boundary-summing regions) to turn them into 4-balls, giving an embedding  $$X_1 \natural \cdots \natural X_k \natural\cdots \natural X_n \ \hookrightarrow \  B^4 \natural \cdots\natural X_k \natural  \cdots \natural B^4,$$
where the target is diffeomorphic to $X_k$. This embedding of $X_1 \natural \cdots \natural X_n$ into $X_k$ induces the projection from $H_2(X_1 \natural \cdots \natural X_n)$ to $H_2(X_k)$, so it carries $S'$ to a smoothly embedded surface in $X_k$ that is still homologically essential. By Proposition~\ref{prop:properties}, it follows that $S'$ has genus at least two, hence so does the original surface $S$.
\end{proof}

It is important to note that the argument above proves an a priori stronger statement than that of Theorem~\ref{thm:relative}, which we record here for use in the final argument for Theorem~\ref{thm:exotic-mfld}.
\begin{theorem}\label{thm:lifttoD}
There is no diffeomorphism $Q(X)\to Q(X')$ that lifts to a diffeomorphism $D(X)\to D(X')$.
\end{theorem}
 
\begin{remark}\label{rem:unicover}%
Let $\xtilde$ (resp.~$\xtilde'$) denote the universal cover of $X$ (resp.~$X'$).
The triangulation of $\partial X$ lifts to a triangulation $\tilde{\mathcal{T}}$ of $\partial \xtilde$. Use this to define a right-angled Coxeter group $\tilde{W}$ and a corresponding Davis complex $D(\xtilde)$. Then $D(\xtilde)$ is the universal cover of $D(X)$. Similarly, we get  $\xtilde'$ and the corresponding universal cover $D(\xtilde')$ of $D(X')$. 
We conjecture that $D(\xtilde)$ and $D(\xtilde')$ are not simply connected at infinity and hence, that neither is homeomorphic to $\R^4$. (In particular, neither is an exotic $\R^4$.) An interesting open question is whether the open contractible manifolds $D(\xtilde')$ and $D(\xtilde)$ are diffeomorphic.
\end{remark}

\section{Characteristic subgroups}\label{sec:characteristic}

 To complete the proof of Theorem~\ref{thm:exotic-mfld} (in light of Theorems~\ref{thm:relative} and \ref{thm:lifttoD}), it suffices to show that \emph{any} potential diffeomorphism $Q(X)\to Q(X')$ would lift to a diffeomorphism $D(X)\to D(X')$. This lifting problem can be recast in terms of the behavior of the subgroup $\pi_1(D(X)) \leq \pi_1(Q(X))$  under automorphisms of $\pi_1(Q(X))$.

\subsection{Flag-no-square complexes and characteristic subgroups} \label{subsec:characteristic} A \emph{cycle} in a simplicial complex $\mathcal{T}$ is a subcomplex homeomorphic to the circle $\mathbb S^1$; its \emph{length} is the number of edges in the cycle. A \emph{diagonal} of a cycle is an edge connecting
any two nonconsecutive vertices in this cycle.
A simplicial complex $\mathcal{T}$ is said to satisfy the \emph{flag-no-square condition} if $\mathcal{T}$ is a flag complex
and if any cycle of length 4 in $\mathcal{T}$ has a diagonal.

\begin{proposition}\cite[Proposition 2.13]{przytycki2009flag}
Let $\mathcal{T}$ be a 3-dimensional simplicial complex. Then it admits a subdivision that is flag-no-square.
\end{proposition}

The following proposition is a consequence of \cite[Theorem 17.1]{Moussong-hyperbolic} (see also \cite[Corollary 12.6.3]{davis:book2}).

\begin{proposition}
\label{prop:noZ2}
Let $\mathcal{T}$ be a simplicial complex that is flag-no-square. Let $W(\mathcal{T})$ be the right-angled Coxeter group associated with $\mathcal{T}$ (as defined in \S\ref{subsec:coxeter}). Then $W(\mathcal{T})$ is word hyperbolic; hence, it does not contain any subgroup that is isomorphic to $\mathbb Z^2$. 
\end{proposition}

In turn, the above results provide a degree of control over $\mathbb{Z}^2$-subgroups of $\pi_1(Q(X))$, enabling us to prove that $\pi_1(D(X))$, which is an infinite free product of copies of $\pi_1(X) \cong \mathbb{Z}^2$, is a characteristic subgroup of $\pi_1(Q(X))$.  

\begin{lemma}\label{lem:characteristic}
Let $X$ be a compact 4-manifold equipped with a flag triangulation $\mathcal{T}$ of $\partial X$. Let $W=W(\mathcal{T})$ be the right-angled Coxeter group associated with $\mathcal T$. Let $D(X)$ be the associated Davis complex as in \S\ref{sec:reflection},  with the natural action $W\curvearrowright D(X)$, and let $Q(X)$ be the quotient of $D(X)$ by a torsion-free finite index subgroup  $W_0$ of $W$. 
If $\pi_1(X)$ is isomorphic to $\mathbb{Z}^2$ and the triangulation $\mathcal T$ of $\partial X$ is flag-no-square, then $\pi_1(D(X))$ is a characteristic subgroup of $\pi_1 (Q(X))$.
\end{lemma}

\begin{proof}
By construction, $D(X)$ is a normal covering of $Q(X)$ with deck transformation group $W_0$. This gives an exact sequence:
\begin{equation}\label{eq:ses1}
1\to \pi_1(D(X))\to \pi_1(Q(X))\to W_0\to 1.
\end{equation}
We claim that if $H\le \pi_1(Q(X))$ is isomorphic to $\mathbb Z^2$, then $H\le \pi_1(D(X))$. Recall that $\pi_1(D(X))$ is a free product of copies of $\pi_1 (X)$, with one copy of $\pi_1(X)$ for each element in $W$ (cf.~\cite[Remark~15.9]{davis:aspherical}). As $\pi_1(X)\cong \mathbb Z^2$, this claim implies that for any automorphism $f$ of $\pi_1(Q(X))$, we have $f(\pi_1(D(X)))\le \pi_1(D(X))$; i.e., $\pi_1(D(X))$ is a characteristic subgroup of $\pi_1(Q(X))$.

It remains to prove the claim. Since the triangulation of $\partial X$ is flag-no-square, Proposition~\ref{prop:noZ2} implies that $W_0$ does not contain a subgroup isomorphic to $\mathbb Z^2$. Let $L$ be the image of $H$ under $\pi_1(Q(X))\to W_0$. Since $W_0$ is torsion-free and word hyperbolic, if $L$ is nontrivial, the only possibility is $L\cong \mathbb Z$. Next we will show that $L$ being isomorphic to $\mathbb Z$  leads to a contradiction, proving the claim. Our strategy will be to consider the action of $L \subset W_0$ on the complementary subgroup $N= H \cap \pi_1(D(X)) \cong \mathbb{Z}$ in the restricted short exact sequence
\begin{equation*}
1\to N \to H \to L \to 1
\end{equation*}
that parallels \eqref{eq:ses1}; the fact that $H$ is abelian constrains the $L$-action on $N$ and leads to a contradiction.

To this end, note that each element of $W_0$ lifts to an element in $\pi_1(Q(X))$, which gives an automorphism of $\pi_1(D(X))$; moreover, different lifts give rise to the same automorphism up to an inner automorphism of $\pi_1(D(X))$. This gives a well-defined homomorphism $\varphi:W_0\to \operatorname{Out}(\pi_1(D(X)))$, where $\operatorname{Out}(\pi_1(D(X)))$ denotes the outer automorphism group of $\pi_1(D(X))$. Later, we will need the following topological description of $\varphi$. Take a base point $p\in D(X)$. Then we can identify $\pi_1(D(X),p)$ with $\pi_1(D(X),q)$ for any $q\neq p$, by choosing a path from $p$ to $q$. This identification is well-defined up to an inner automorphism of $\pi_1(D(X),p)$. Given an element $\bar g\in W_0$,  the action $W_0\curvearrowright D(X)$ by deck transformations gives an isomorphism $\bar g_*:\pi_1(D(X),p)\to\pi_1(D(X),\bar g(p))$. As we can identify $\pi_1(D(X),\bar g(p))$ with $\pi_1(D(X),p)$, the isomorphism $\bar g_*$ gives an element in the outer automorphism group of $\pi_1(D(X),p)$, which is exactly $\varphi(\bar g)$ for the map $\varphi$ defined above.

Now suppose $L\cong \mathbb Z$, and let $\bar g\in L$ be a generator of $L$. Let $N=\pi_1(D(X))\cap H\cong \mathbb{Z}$ be defined as before. 
 Let $g\in H\le \pi_1(Q(X))$ be a lift of $\bar g$. Since $H$ is abelian, we have $g h g^{-1}=h$ for all $h \in N \le H$. So, conjugation by $g$ gives an automorphism $\theta_g:\pi_1(D(X))\to \pi_1(D(X))$ that restricts to the identity on $N$.
 The element in the outer automorphism group of $\pi_1(D(X))$ represented by $\theta_g$ is exactly $\varphi(\bar g)$, where $\varphi:W_0\to \operatorname{Out}(\pi_1(D(X)))$ was defined in the previous paragraph. In particular, $\varphi(\bar g)$ fixes the $\pi_1(D(X))$-conjugacy class of each element of $N$.
  (Note that the action of $\operatorname{Out}(\pi_1(D(X)))$  is only well-defined on the conjugacy classes of elements in $\pi_1(D(X))$. We have a particular lift of $\varphi(\bar g)$ to $\operatorname{Aut}(\pi_1(D(X)))$, namely conjugation by $g$, that acts as the identity on each element of $N$, hence the associated outer  automorphism $\varphi(\bar g)$ acts as the identity on the $\pi_1(D(X))$-conjugacy class of $N$.)
 
Now let $\alpha$ be a loop based at $p$ that represents a generator of $N$ in $\pi_1(D(X),p)$. Then the previous paragraphs imply that for any path $\beta$ from $p$ to $\bar g(p)$, the loop $\beta\bar g(\alpha)\beta^{-1}$ gives an element in $\pi_1(D(X),p)$ that is conjugate to $[\alpha]$ by an element of $\pi_1(D(X),p)$. Here $\beta^{-1}$ denotes the inverse path of $\beta$. In particular, this  implies that $\alpha$ and $\bar g(\alpha)$ are freely homotopic in $D(X)$. Thus $\alpha$ and $\bar g^m(\alpha)$ are freely homotopic in $D(X)$ for any $m\ge 1$. We will show below that this leads to a contradiction.

Take a fundamental domain $X_1\cong X$ for the action of $W$ on $D(X)$. As in \S\ref{subsec:davis space}, we choose an enumeration $g_1=\operatorname{id},g_2,g_3,\dots$ of elements of $W$ such that $\ell(g_i)\le \ell(g_j)$ whenever $i\leq j$.
Let $X_1=P_1\subset P_2\subset P_3\subset \cdots$ be the exhaustion of $D(X)$ defined in \S\ref{subsec:davis space} with $P_n=\bigcup_{1\le i\le n}g_i X_1$. Since the image of $\alpha$ is compact, there exists $n_0$ so that its image is contained in $P_{n_0}$. Since the  group action $W\curvearrowright D(X)$ is proper and $\bar g$ has infinite order in $W$, we know that there is an $m_0>0$ such that 
\begin{equation}
\label{eq:disjoint}
    P_{n_0}\cap \bar g^{m_0}(P_{n_0})=\emptyset.
\end{equation}

Let $Y\subset X$ be a subset obtained by removing a collar neighborhood of $\partial X$ in $X$ (homeomorphic to $\partial X\times [0,1)$) from $X$. We collapse $Y\subset X$ to a point and obtain the topological space $\bar X$, which is homeomorphic to a cone over $\partial X$. Let $Y_1\subset X_1$ be the subspace of $X_1$ arising from $Y\subset X$. For each $i>n_0$, we collapse $g_i Y_1$ in $D(X)$ to a point. This gives a new topological space $\bar D(X)$, with $\pi: D(X)\to \bar D(X)$ being the natural continuous map. As $\alpha$ and $\bar g^{m_0}(\alpha)$ are freely homotopic in $D(X)$, we know $\pi(\alpha)$ and $\pi(\bar g^{m_0}(\alpha))$ are freely homotopic in $\bar D(X)$. In what remains, we will show $\pi(\bar g^{m_0}(\alpha))$ is null-homotopic in $\bar D(X)$, but $\pi(\alpha)$ is homotopically nontrivial in $\bar D(X)$, which gives the desired contradiction.

Let $\bar P_n=\pi(P_n)$. As the procedure for obtaining $\bar D(X)$ from $D(X)$ does not change the boundary of each chamber, we know from Lemma~\ref{lem:disk} that $\bar P_n\cap \pi(g_{n+1} X_1)$ is a codimension-zero disk in $\partial \bar P_n$. Hence, the van Kampen theorem implies that $\bar P_n\to \bar P_{n+1}$ is $\pi_1$-injective for each $n$. (For $n \geq n_0$, this map is in fact $\pi_1$-bijective, since $\bar P_{n+1}$ is the union along a disk of $\bar P_n$ with a contractible space.) This implies that $\bar P_n\to \bar D(X)$ is $\pi_1$-injective for each $n$.
Since $\alpha$ is homotopically nontrivial in $D(X)$, we know it is homotopically nontrivial in $P_{n_0}$. As $\pi$ restricted to $P_{n_0}$ is a homeomorphism onto $\bar P_{n_0}$, we know $\pi(\alpha)$ is homotopically nontrivial in $\bar P_{n_0}$. Hence, $\pi(\alpha)$ is homotopically nontrivial in $\bar D(X)$.

Next, consider $\pi(\bar g^{m_0}(\alpha))$. By \eqref{eq:disjoint}, each chamber in $\bar g^{m_0}(P_{n_0})$ is collapsed under $\pi$. We apply the Davis construction to $\bar X$ to obtain $D(\bar X)$. Note that $D(\bar X)$ has a similar exhaustion, denoted by $R_1\subset R_2\subset \cdots$. As $\bar X$ (that is, the cone on $\partial X$) is simply-connected, we know $R_n$ is simply-connected for each $n$ by the same argument as the previous paragraph. By construction, $\pi(\bar g^{m_0}(P_{n_0}))$ is homeomorphic to $R_{n_0}$, and hence, also is simply-connected. It follows that $\pi(\bar g^{m_0}(\alpha))$ is null-homotopic in $\pi(\bar g^{m_0}(P_{n_0}))$, and hence, is null-homotopic in $\bar D(X)$, as desired.
\end{proof}

\subsection{Conclusion} We can now complete the proof of our main result. 

\begin{proof}[Proof of Theorem~\ref{thm:exotic-mfld}]
    By Theorem~\ref{thm:relative}, there is a homeomorphism $f: Q(X) \to Q(X')$ which,  by construction, induces a map $f_*$ carrying $\pi_1(D(X))\leq \pi_1(Q(X))$ isomorphically to $\pi_1(D(X'))\leq \pi_1(Q(X'))$. 
    
    Next suppose there is a diffeomorphism $g: Q(X) \to Q(X')$. By Theorem~\ref{thm:lifttoD}, to obtain a contradiction, it suffices to show that this lifts to a diffeomorphism $\tilde g: D(X) \to D(X')$. Such a lift exists if and only if $g_*(\pi_1(D(X)))$ lies inside $\pi_1(D(X')) \leq \pi_1(Q(X'))$. To show that this condition is met, note that $f_*^{-1}\circ g_*$ is an automorphism of $\pi_1(Q(X))$, and hence, preserves the characteristic subgroup $\pi_1(D(X))$ by Lemma~\ref{lem:characteristic}. Since $f_*^{-1}$ restricts to an isomorphism between $\pi_1(D(X'))$ and $\pi_1(D(X))$, it follows that $g_*$ must carry $\pi_1(D(X))$ to $\pi_1(D(X'))$, as desired.     
\end{proof}

\begin{remark}\label{rem:infinite}
\textbf{(a)} We can obtain infinitely many distinct pairs of exotic aspherical closed 4-manifolds starting from the same fixed pair of input 4-manifolds $X$ and $X'$ and a fixed flag-no-square triangulation $\mathcal T$ of $\partial X = \partial X'$ as follows. Recall from Section~\ref{subsec:quotient} that the construction of the closed manifold $Q(X)$ depends on the choice of a finite index torsion-free subgroup $W_0$ of the right-angled Coxeter group $W$ associated to $\mathcal T$. Choose an infinite strictly descending chain $W_1\supsetneq W_2\supsetneq W_3\supsetneq \cdots$ of finite index torsion-free subgroups of $W$; as any right-angled Coxeter group is residually finite, we can arrange such a chain such that $\cap_{i=1}^\infty W_i=\{1\}$. Applying the construction of \S\ref{subsec:quotient}, we get infinitely many exotic aspherical pairs. Indeed, writing $Q_i(X)$ (resp.~$Q_i(X')$) for the aspherical manifolds obtained from $W_i$ using $X$ (resp.~$X'$) as chamber, then $Q_i(X)$ and $Q_i(X')$ are an exotic pair.
 
Moreover, $Q_i(X)$ and $Q_j(X)$ are not homeomorphic when $i\neq j$, otherwise we would have an isomorphism between $\pi_1(Q_i(X))$ and $\pi_1(Q_j(X))$, which preserves the $\pi_1(D(X))$ subgroup by the same argument as Lemma~\ref{lem:characteristic}. This would give an isomorphism between $W_i$ and $W_j$. However, as $W_i$ is a one-ended torsion-free hyperbolic group, it is not isomorphic to any proper subgroup of itself \cite{sela1997structure}.

\textbf{(b)} Alternatively, by applying our arguments to the 4-manifolds $X_m$ and $X_m'$ described in Remark~\ref{rem:extend}, we can produce infinitely many exotic pairs of closed aspherical 4-manifolds $Q(X_m)$ and $Q(X_m')$. The lifting arguments of \S\ref{subsec:characteristic} apply directly, so any diffeomorphism $Q(X_m) \to Q(X_m')$ would lift to a diffeomorphism $D(X_m) \to D(X_m')$. However,  arguing as in the proof of Theorem~\ref{thm:relative} shows that $D(X_m')$ contains a homologically essential, smoothly embedded torus while $D(X_m)$ has no homologically essential, smoothly embedded surface of genus less than $2 + m$. 

Moreover, the same argument shows that $Q(X_m)$ and $Q(X_n)$ are not diffeomorphic when $m \gg n$. Indeed, the proof of Lemma~\ref{lem:characteristic} implies that any diffeomorphism $Q(X_m) \to Q(X_n)$ would induce a map on $\pi_1$ carrying $\pi_1(D(X_m))$ into $\pi_1(D(X_n))$, and hence would lift to a diffeomorphism $D(X_m) \to D(X_n)$, contradicting the genus bounds above.  While we expect that varying $m$ and $n$ always yields pairs that are not homeomorphic, we do not address this question here.
\end{remark}

\bibliographystyle{alphaurl}

\bibliography{biblio}

 \end{document}